\documentclass[12pt]{amsart}
\usepackage{amsmath,amssymb}
\usepackage{graphicx,bm}
\usepackage[margin=1in]{geometry}
\usepackage{hyperref}

\numberwithin{equation}{section}

\newtheorem{thm}{Theorem}[section]
\newtheorem{cor}[thm]{Corollary}
\newtheorem{lem}[thm]{Lemma}
\newtheorem{prop}[thm]{Proposition}

\newtheorem{defn}[thm]{Definition}
\newtheorem{rmk}[thm]{Remark}
\newtheorem{clm}[thm]{Claim}

\newcommand{\pt}{\partial}

\begin{document}

\title[Quaternionic Monge-Amp\`ere Equations]{Dirichlet Problem of Quaternionic\\ Monge-Amp\`ere Equations }
\author{Jingyong ZHU}
\address{School of Mathematical Sciences,  University of Science and Technology of China}
\email{zjyjj14@mail.ustc.edu.cn}


\begin{abstract}
  In this paper, the author studies quaternionic Monge-Amp$\grave{e}$re equations and obtains the existence and uniqueness of the solutions to the Dirichlet problem for such equations without any restriction on domains. Our paper aims to answer the question proposed by Semyon Alesker in \cite{alesker2003quaternionic}. It also extends relevant results in \cite{caffarelli1985dirichlet} to the quaternionic vector space.

\end{abstract}
\maketitle
\section{Introdution}    
  Quaternion and HKT-geometry is an important branch of maths. Mathematicians have discoveried some interesting facts from it. It has many applications in mathematical physics. Recently, the question whether there is a quaternionic version of Calabi-Yau Theorem has attracted some experts to do
 research on it, and they have obtained some results \cite{alesker2013solvability}\cite{alesker2013uniform}\cite{alesker2010quaternionic}. Relating to this problem, Dirichlet problem for quaternionic Monge-Amp$\grave{e}$re(MA) equations on
 arbitrary strictly pseudoconvex bounded domains is an open problem\cite{alesker2003quaternionic}.In this paper, we solve this issue.

  To begin with, we want to describe the background. The classical solvability of the Dirichlet problem for real and complex Monge-Amp$\grave{e}$re equations were proved under the
 condition of convexity and pseudo-convexity of domains in \cite{caffarelli1984dirichlet} and \cite{caffarelli1985dirichlet}, respectively. To general domains in $\mathbb{C}^n$, Bo Guan managed to obtain the same result in \cite{caffarelli1985dirichlet} assuming the existence of a subsolution to the corresponding equation\cite{guan1998dirichlet}, and generalized the result to totally real submanifolds\cite{guan2010complex} and Hermitian manifolds\cite{guan2013dirichlet}\cite{guan2013class}. In \cite{caffarelli1985dirichlet}, L. Caffarelli and his co-authors created a subsolution to the Dirichlet problem using the defining function for the strongly pseudo-convex domain. In \cite{guan2002extremal}, P. Guan constructed a subsolution to the Dirichlet problem for a degenerate complex Monge-Amp\`ere equation on a special domain with some pieces of the boundary being concave. The Monge-Amp\`ere equation also has many geometric applications, for example, Calabi conjecture\cite{calabi1954space}. In \cite{yau1978ricci}, Yau solved the Calabi conjecture. Yau's work also shows the existence of K$\ddot{a}$hler-Einstein metrics on K$\ddot{a}$hler manifolds with nonpositive first Chern class.

  In \cite{alesker2003quaternionic}, S. Alesker proved a result on existence and uniqueness of the smooth solution of Dirichlet problem
 \begin{equation}\label{eq11}
\begin{cases}
  \mathrm{det}\left(\frac{\pt^2u}{\pt q_i\pt \overline{q}_j}\right)=f(q)\ \ \ \text{in} \ \mathbf{B},  \\
  u|_{\pt{\mathbf{B}}}=\varphi, \ \ \
\end{cases}
\end{equation}
 where $q\in\mathbf{B}$ and $\mathbf{B}$ is the Euclidean ball in $\mathbb{H}^n$ which denotes the space of n-tuples of quaternions $(q_1,\ldots,q_n)$. He mainly followed the method in \cite{caffarelli1985dirichlet}, but he made a strong restriction on the domain. He said the reason why he failed to solve \eqref{eq11} on general strictly pseudoconvex bounded domains is the fact that the class of diffeomorphisms preserving the class of quaternionic plurisubharmonic(psh) function must be affine transformations. In a word, the priori estimates in both \cite{alesker2003quaternionic} and \cite{caffarelli1985dirichlet} depend on the positive definiteness of the matrix in the local expression of the boundary, while this positiveness can't be preserved all the time on quaternionic strictly pseudoconvex domains in general. However, from the statement in the preceding paragraph, we know this positiveness is not necessary. Since the field of quaternions is non-commutative, we need to modify some auxiliary functions in \cite{guan2010complex} and the proof of (1.47) in \cite{caffarelli1985dirichlet}. At last, we obtain the following:
\begin{thm}\label{thm11}
 Let $\Omega$ be any bounded domain with a smooth boundry. If there exists a psh function $\underline{u}\in {C^\infty}(\Omega)$ be a subsolution such that
 \begin{equation}
\begin{cases}
  \mathrm{det}\left(\frac{\pt^2\underline{u}}{\pt q_i\pt {\overline{q}_j}}\right)\geq f(q,\underline{u}),\ \ \text{in } \Omega\\
  \underline{u}|_{\pt{\Omega}}=\varphi\in {C^\infty}(\pt\Omega), \ \ \
\end{cases} 
\end{equation}
where $f\in {C^\infty}(\overline{\Omega}\times\mathbb{R})$ and
$f>0,f_u=\frac{\pt{f}}{\pt{u}}\geq0$ for any
$q\in\overline\Omega,u\in\mathbb{R}$. Then there exists an unique
psh function $u\in {C^\infty}(\overline\Omega)$ solving
 \begin{equation}\label{eq:13}
\begin{cases}
  \mathrm{det}(u_{i\overline{j}})=\mathrm{det}\left(\frac{\pt^2u}{\pt q_i\pt {\overline{q}_j}}\right)=f(q,u),\ \ \text{in } \Omega\\
  u|_{\pt{\Omega}}=\varphi\in {C^\infty}(\pt\Omega), \ \ \
\end{cases}
\end{equation}.
\end{thm}  

  Proceeding as in \cite{caffarelli1985dirichlet}, we establish the relation between the convexity of domain in $\mathbb{H}^n$ and the subsolution to the Dirichlet problem of quaternionic MA equations:
\begin{prop}\label{prop12}
Suppose that $\Omega$ is a
quaternionic strictly pseudoconvex bounded domain. For any
$q\in\overline\Omega,u\in\mathbb{R},p\in\mathbb{R}^{4n}$, assume
$f\in C^{\infty}(\overline\Omega\times\mathbb{R}\times\mathbb{R}^{4n})$
satisfies
\begin{equation}
f>0,\ \ f_u(q,u,\nabla{u})\geq0,\ \  |f_{p_i}(q,u,p)|\leq
Cf^{1-\frac{1}{n}},
\end{equation} where $C$ is a constant. Then there
exists a subsolution $\underline{u}\in C^\infty(\overline\Omega)$
such that
 \begin{equation}\label{eq15}
\begin{cases}
  \mathrm{det}\left(\frac{\pt^2\underline{u}}{\pt q_i\pt {\overline{q}_j}}\right)\geq f(q,\underline{u},\nabla\underline{u}),\ \ \text{in } \Omega\\
  \underline{u}|_{\pt{\Omega}}=\varphi\in {C^\infty}(\pt\Omega), \ \ \
\end{cases}
\end{equation}
where $p=\nabla{u}$.
\end{prop}

From this property, we have                
\begin{cor} \label{cor13}
 Assume that $\Omega$ is a quaternionic
strictly pseudoconvex bounded domain. If $f>0$ and $f_u\geq0$, then
there exists an unique psh function $u\in
{C^\infty}(\overline\Omega)$ solving $\eqref{eq:13}$.
\end{cor}  

\begin{rmk} The restriction $``f_u\geq0"$ above is necessary
for the uniqueness of the solution, and Corollary \ref{cor13} settles the
``Question 4" in \cite{alesker2003quaternionic}.
\end{rmk}

 This article is organized as follows. In section 2, we review
some basic definitions and facts from the theory on non-commutative
determinants and plurisubharmonic functions of quaternionic
variables. In section 3, we prove priori estimates up to second
order excepting the boundary estimates for the second order normal
derivatives, which are proved more detailedly in section 4. In
section 5, we construct a subsolution to clarify Proposition \ref{prop12}.

\section{Quaternionic linear algebra}
 In the whole of this article, we consider $\mathbb{H}^n$ as right $\mathbb{H}$-vector space, i.e. vectors are multiplied by scalars on the right. The
standard thoery of vector spaces, basis, and dimension works over
$\mathbb{H}^n$, exactly like in the commutative case. However, the
theory of non-commutative determinants is quite different and
complicated. There are many useful determinants over $\mathbb{H}$,
e.g. Dieudonn$\mathrm{\acute{e}}$ determinant. Experts are still
searching for the best determinant which preserves most of the
identities and inequalities known for usual determinant of real and
complex matrices. We know the importance of real symmetric and
complex Hermitian matrices. Over $\mathbb{H}$, similarly, there is a
class of quaternionic matrices called hyperhermitian.
\begin{defn}\cite{alesker2003non} Let V be a right $\mathbb{H}$-vector space. A hyperhermitian semilinear form on V is a map $a: V\times V\rightarrow\mathbb{H}$
 satisfying the following properties:\\
 (a)\ $a$ is additive with respect to each argument;\\
 (b)\ $a(x,y\cdot q)=a(x,y)\cdot q$ for any $x,y\in V$ and $q\in\mathbb{H}$;\\
 (c)\ $a(x,y)=\overline{a(y,x)}$,\\
 where $\overline{q}$ denotes the usual quaternionic conjugation of $q\in\mathbb{H}$.
 \end{defn}

 An $n\times n$ quaternionic matrix $A=(a_{ij})$ is called $hyperhermitian$ if $A^*=A$, i.e.,$a_{ij}=\overline{a_{ji}}$ for all $i,j$. The proposition below points out the relation between hyperhermitian semilinear forms and hyperhermitian matrices.
\begin{prop}\cite{alesker2003non} Fix a basis in a
finite dimensional right quaternionic vector space V. Then there is
a natural bijection between the space of hyperhermitian semilinear
form on V and the sapce $\mathcal{H}_n$ of $n\times n$
hyperhermitian matrices.
\end{prop}

 We are going to state some basic facts about hyperhermitian matrices as follows.
\begin{prop}\cite{alesker2003non} \label{prop23}Let A be a matrix of a given hyperhermitian form in a given basis. Assume that C is transition matrix from this basis to another one. Then we have $$A'=C^*AC,$$
where $(C^*)_{ij}=\overline{C}_{ji}$ and $A'$ denotes the matrix of the given form in the new basis.
\end{prop}
\begin{rmk}\cite{alesker2003non} The matrix $C^*AC$ is hyperhermitian for any hyperhermitian matrix A and any matrix C. In particular $C^*C$ is always hyperhermitian.
\end{rmk}
\begin{defn}\cite{alesker2003non} A hyperhermitian
 semilinear  form $a$ is called $positive$ $definite$ if $a(x,x)>0$
 for any non-zero vector $x$. Similarly, $a$ is called
 $non\text{-}negative$ $definite$ if $a(x,x)\geq0$ for any vector
 $x$.
 \end{defn}

  Let us fix on our quaternionic right space V a positive
definite hyperhermitian form $(\cdot,\cdot)$. The space with such a
form is called $hyperhermitian$ $space$. For any quaternionic linear
operator $\phi: V\rightarrow V$ in hyperhermitian space one can
define the adjoint operator $\phi^*: V\rightarrow V$ in the usual
way, i.e. $(\phi x,y)=(x,\phi y)$ for any $x,y\in V$. Moreover, if
one fixes an orthonormal basis in the space V, then the operator
$\phi$ is selfadjoint if and only if its matrix under this basis is
hyperhermitian.
\begin{prop}\cite{alesker2003non}\label{prop26} For any
selfadjoint operator in a hyperhermitian space there exists an
orthonormal basis such that its matrix in this basis is diagonal and
real.
\end{prop}

 We are going to introduce the Moore determinant of hyperhermitian matices. For the definition of the Moore determinant one can refer to \cite{moore1922determinant}\cite{aleksandrov1938gemischte}\cite{alesker2003non}. It is useful to give an explicit formula for the Moore determinant. Let $A=(a_{ij})^n_{i,j=1}$ be a hyperhermitian $(n\times n)$-matrix. Suppose that $\sigma$ be a permutation of {1,...,n}. Write $\sigma$ as a product of disjoint cycles such that each cycle starts with the smallest number. Since disjoint cycles commute we can write $$\sigma=(k_{11}\ldots k_{1j_1})(k_{21}\ldots k_{2j_2})\ldots(k_{m1}\ldots k_{mj_m}),$$
where for each i we have $k_{i1}<k_{ij}$ for all $j>1$, and $k_{11}>k_{21}>\ldots>k_{m1}$. This expression is unique. Let $sgn(\sigma)$ be the parity of $\sigma$. For the next formula one can refer to \cite{moore1922determinant}\cite{aslaksen1996quaternionic}.
\begin{thm}\cite{moore1922determinant}\cite{aslaksen1996quaternionic} The Moore determinant of A is
$$\mathrm{det}A=\sum_{\sigma}sgn(\sigma)a_{k_{11}k_{12}}\ldots a_{k_{1j_1}k_{11}}a_{k_{21}k_{22}}\ldots a_{k_{mj_m}k_{m1}},$$
where the sum runs over all permutations.
\end{thm}

  From now on, we denote the Moore determinant of A by $\rm{det}A$. For hyperhermitian matrices, the Moore determinant is the best one, because it has almost all the algebraic and analytic properties of the usual determinant of real symmetric and complex hyperhermitian matrices. Let us state some of them.
\begin{thm}\label{thm28}\cite{alesker2003quaternionic} (1) The Moore determinant of any complex hermitian matrix considered as quaternionic hyperhermitian matrix is equal to its usual determinant.
 (2) For any hyperhermitian matrix A and any matrix C we have $$\mathrm{det}(C^*AC)=\mathrm{det}A\cdot\mathrm{det}(C^*C).$$
\end{thm}
\begin{defn}\label{def29} \cite{alesker2003non} Let
$A=(a_{ij})_{i,j=1}^n$ be a quaternionic hyperhermitian matrix. A is
called $non\text{-}negative$ $definite$ if for every n-column of
quaternions $\xi=(\xi_i)^n_{i=1}$ one has
$$\xi^*A\xi=\sum\overline{\xi_i}a_{ij}\xi_j\geq0.$$
where $\sum$ denotes a summation over repeated indices. Similarly, A
is called $positive$ $definite$ if the above expression
is strictly positive unless $\xi=0$.
\end{defn}

  From Proposition \ref{prop23}, Proposition \ref{prop26} and Theorem \ref{thm28}, one can easily check:
\begin{prop}\cite{alesker2003non} Let A be a
non-negative(resp. positive) definite hyperhermitian matrix. Then
$\mathrm{det}A\geq0$(resp. $\mathrm{det}A>0$).
\end{prop}

  The following theorem is a quaternionic generalization of the standard Sylvester criterion.
\begin{thm}\cite{alesker2003non} A hyperhermitian
$(n\times n)$-matrix A is positive definite if and only if the Moore
determinants of all the left upper minors of A are positive.
\end{thm}

 Let us define now the mixed discriminant of hyperhermitian matrices in analogy with the case of real symmetric matrices studied by A. D. Aleksandrov\cite{aleksandrov1938gemischte}.
\begin{defn}\cite{alesker2003non} Let $A_1,\ldots,A_n$ be hyperhermitian $(n\times n)$-matrices. Consider the homogeneous polynomial in real variables $\lambda_1,\ldots,\lambda_n$ of degree n equal to $\mathrm{det}(\lambda_1A_1+\cdots+\lambda_nA_n)$. The coefficient of the monomial $\lambda_1\cdot\cdots\cdot\lambda_n$ divided by $n!$ is called the $mixed$ $discriminant$ of the matrices $A_1,\ldots,A_n$, and it is denoted by $\mathrm{det}(A_1,\ldots,A_n)$.
\end{defn}
\begin{prop}\cite{alesker2003non} The mixed
discriminant is symmetric with respect to all variables, and linear
with respect to each of them, i.e.
$$\mathrm{det}(\lambda A'_1+\mu A''_1,A_2,\ldots,A_n)=\lambda\cdot\mathrm{det}(A'_1,A_2,\ldots,A_n)+\mu\cdot\mathrm{det}(A''_1,A_2,\ldots,A_n)$$
for any real $\lambda,\mu$. In particular,
$\mathrm{det}(A,\ldots,A)=\mathrm{det}A$.
\end{prop}

 By Proposition \ref{prop23}, Proposition \ref{prop26} and Theorem \ref{thm28}, we get the following algebraic identity.
\begin{clm}\cite{alesker2003quaternionic}\label{clm214} For any vector $a=(a_1,\cdots,a_n)$ we have
$${\rm{det}}((a_{\overline{j}}u_{i\overline{n}})+(a_{\overline{j}}u_{i\overline{n}})^*,\pt^2u[n-1])
=2({\rm{Re}}\ a_{\overline{n}}){\rm{det}}(u_{i\overline{j}}),$$
where $(u_{i\overline{j}})$ is the matrix in Theorem \ref{thm11}.
\end{clm}

 The Theorem below is very useful in our proof.
\begin{thm}\cite{alesker2003non}\label{thm215} (1) The mixed discriminant of positive(resp. non-negative) definite matrices is positive(resp. non-negative).
 (2) Fix positive definite hyperhermitian $(n\times n)$-matrices $A_1,\ldots,A_{n-2}$. On the real linear space of hyperhermitian $(n\times n)$-matrices consider the bilinear form $$B(X,Y):=\mathrm{det}(X,Y,A_1,\ldots,A_{n-2}).$$ Then B is non-degenerate quadratic form, and its signature has one plus and the rest are minueses.
 \end{thm}
\begin{cor}\cite{alesker2003non}[Aleksandrov inequality] Let $A_1,\ldots,A_{n-1}$ be positive definite hyperhermitian $(n\times n)$-matrices. Then for any hyperhermitian matrix X we have
$${\mathrm{det}(A_1,\ldots,A_{n-1},X)}^2\geq\mathrm{det}(A_1,\ldots,A_{n-1},A_{n-1})\cdot\mathrm{det}(A_1,\ldots,A_{n-2},X,X),$$
and the equality is satisfied if and only if the matrix X is
proportional to $A_{n-1}$.
\end{cor}

 For any $\varepsilon>0$, applyinng Theorem \ref{thm215} to $(\varepsilon X+\frac{1}{\varepsilon}Y)(\varepsilon X+\frac{1}{\varepsilon}Y)^*$, we have
  \begin{cor}\label{cor217} For a fixed $n\times n$ positive definite hyperhermitian matrix $A$ and any two $(n\times n)$-matrices $X,Y$, we have
$$|\mathrm{det}(XY^*+YX^*,A[n-1])|\leq\varepsilon^2\mathrm{det}(XX^*,A[n-1])+\frac{1}{\varepsilon^2}\mathrm{det}(YY^*,A[n-1]).$$
\end{cor}

 In the rest of this section, we want to recall some basic definitions and facts from the theory of psh functions of quaternionic variables in \cite{alesker2003non}\cite{alesker2003quaternionic}.
\begin{defn}\cite{alesker2003quaternionic} Let $\Omega$ be a bounded domain in $\mathbb{H}^n$. A real valued function $u: \Omega\rightarrow\mathbb{R}$ is called quaternionic $plurisubharmonic(psh)$ if it is upper semi-continuous and its restriction to any right quaternionic line is subharmonic. In particular, we call a $C^2$-smooth function $u: \Omega\rightarrow\mathbb{R}$ to be $strictly$ $plurisubharmonic(spsh)$ if its restriction to any right quaternionic line is strictly subharmonic(i.e., the Laplacian is strictly positive).
\end{defn}

\begin{defn}\cite{alesker2003quaternionic} An open
bounded domain $\Omega\subset\mathbb{H}^n$ with a smooth boundary
$\pt\Omega$ is called $strictly$ $pseudoconvex$ if for every point
$z_0\in\pt\Omega$ there exists a neighborhood $\mathcal{O}$ and a
smooth strictly psh function h on $\mathcal{O}$ such that
$\Omega\bigcap\mathcal{O}={h<0}, h(z_0)=0$, and
$\nabla{h}(z_0)\neq0$.
\end{defn}

 We usually write a quaternion in the following form $$q=t+x\cdot i+y\cdot j+z\cdot k,$$ where $t,x,y,z$ are real numbers, and $i,j,k$ satisfy the usual relations $$i^2=j^2=k^2=-1,\ \ \ ij=-ji=k,\ \ \ jk=-kj=i,\ \ \ ki=-ik=j.$$
 The Dirac-Weyl(or Cauchy-Riemann) operator $\frac{\pt}{\pt\overline{q}}$ is defined as follows. For any $\mathbb{H}$-valued function f
$$\frac{\pt}{\pt\overline{q}}f:=\frac{\pt f}{\pt t}+i\frac{\pt f}{\pt x}+j\frac{\pt f}{\pt y}+k\frac{\pt f}{\pt z}.$$
 Let us also define the operator $\frac{\pt}{\pt q}$:
$$\frac{\pt}{\pt q}f:=\overline{\frac{\pt}{\pt\overline{q}}\overline{f}}=\frac{\pt f}{\pt t}-\frac{\pt f}{\pt x}i-\frac{\pt f}{\pt y}j-\frac{\pt f}{\pt z}k.$$
 In the case of several quaternionic variables, it is easy to check that those two operators above are commutative. For any real valued twice continuously differentiable function f, the marix $\left(\frac{\pt^2 f}{\pt q_i\pt{\overline{q}_j}}\right)(q)$ is hyperhermitian and we also have another way to define psh function:
\begin{prop}\cite{alesker2003non} Let $f\in C^2(\Omega)$ be real valued. f is quaternion psh if and only if at every point $q\in\Omega$ the matrix $\left(\frac{\pt^2 f}{\pt q_i\pt{\overline{q}_j}}\right)(q)$ is non-negative definite.
\end{prop}

 Last, we want to state the $minimum$ $principle$ to end this section.
\begin{thm}\cite{alesker2003non} Let $\Omega$ be a bounded open set in $\mathbb{H}^n$. If $u,v$ are continuous functions on $\overline\Omega$ which are psh in $\Omega$ and satisfy that $$\mathrm{det}\left(\frac{\pt^2 u}{\pt q_i\pt\overline{q}_j}\right)\leq\mathrm{det}\left(\frac{\pt^2 v}{\pt q_i\pt\overline{q}_j}\right) in \Omega.$$ Then $$\mathrm{min}\{u(z)-v(z)|z\in\overline\Omega\}=\mathrm{min}\{u(z)-v(z)|z\in\pt\Omega\}.$$
\end{thm}

\section{$C^1$ estimates and partial $C^2$ estimates}
 In order to use the continuity method, it is well known that it suffices to prove priori estimates up to the second-order. We will take three steps to achieve this goal.
\\[8pt]                
{\bf  Step 1\ \ Reduce the global 1st-order priori estimates to the
boundary ones. }\\

 First, let us state the main thoerem in step 1.
\begin{thm}\label{thm31} Suppose that a psh function $u\in
C^2(\overline\Omega)$ satisfies (1.3). Then $$\|u\|_{C^1}\leq C,$$
where the constant C only depends on $\|f\|_{C^0}$, $\|f\|_{C^1}$ and
$\Omega$.
\end{thm}

 To prove Theorem \ref{thm31}, we need the following important lemma:
\begin{lem}\label{lem32} Let $D$ be a first-order differential
operator of the form $D=\frac{\mathrm{d}}{\mathrm{d}x_i}$, where $x_i$
is one of the real coordinate axes in $\mathbb{H}^n$. Then we have
$$\max_{\overline{\Omega}}|Du|\leq\max_{\pt\Omega}|Du|+C,$$
where C is a contant dedpending only on $\|f\|_{C^0}$, $\|f\|_{C^1}$
and $\Omega$.
\end{lem}
\begin{proof} Let $L$ be the linearization of the operator $v\mapsto\log(\det(\pt^2v))-\log{f(q,v)}$ at $u$. Explicitly we can write this operator
$$Lv=nf^{-1}\det(\pt^2v,\pt^2u[n-1])-f^{-1}f_{u}v\triangleq (L_0-f^{-1}{f_u})v.$$
Consider the function $\psi=\pm Du+e^{\lambda|q|^2}$, with $\lambda\gg0$ to be determined. One can easily check
$$L_0(Du)=nf^{-1}\det(\pt^2(Du),\pt^2u[n-1])=f^{-1}D(\det(\pt^2u))=D(\log f)$$
and
\begin{align*}
L_0(e^{\lambda|q|^2})&=nf^{-1}\det\left(\left(e^{\lambda|q|^2}\lambda\delta_{ij}+\lambda^2e^{\lambda|q|^2}\overline{q}_iq_j\right),\pt^2u[n-1]\right)\\
&\geq\lambda{e^{\lambda|q|^2}}nf^{-1}\det(I,\pt^2u[n-1])=\lambda{e^{\lambda|q|^2}}\sum{u^{i\overline{i}}},
\end{align*}
where the matrix $(u^{ij})$ and $I$ denote the inverse of the matrix $(u_{ij})$ and the identity matrix, respectively.  Then we have
\begin{align*}
L\psi&=(L_0-f^{-1}{f_u})(\pm Du+e^{\lambda|q|^2})\\
&=\pm L_0(Du)\mp f^{-1}{f_u}Du+L_0(e^{\lambda|q|^2})-f^{-1}{f_u}e^{\lambda|q|^2}\\
&=\pm{D(\log{f})}\mp{f^{-1}}{f_u}Du+L_0(e^{\lambda|q|^2})-f^{-1}{f_u}e^{\lambda|q|^2}\\
&\geq -(\min f)^{-1}{{\|f\|}_{C^1}}+L_0(e^{\lambda|q|^2})-f^{-1}{f_u}e^{\lambda|q|^2}\\
&\geq -C+e^{\lambda|q|^2}(\lambda\sum{u^{i\overline{i}}}-f^{-1}{f_u})\\
&\geq -C+e^{\lambda|q|^2}(n\lambda{f^{-\frac{1}{n}}}-f^{-1}{f_u}).
\end{align*}
 Since $u$ is controlled by the barrirers $h$ and $\underline{u}$, which we will discuss in the proof of Theorem \ref{thm31}, i.e. $$\underline{u}\leq u\leq h,$$ we get $f(q,u)>0$ and $f_u$ are bounded on the bounded domain $\Omega\times [\min\underline{u},\max h]$. Then we can choose a large $\lambda$ to make the last expression positive.
For such a $\lambda$, by maximum principle, the function $\psi$
achieves its maximum on the boundary $\pt{\Omega}$. This proves
Lemma \ref{lem32}.
\end{proof}
\begin{proof}[Proof of Theorem \ref{thm31}]
By Lemma \ref{lem32}, it is sufficient to
estimate $\max\limits_{\pt{\Omega}}|Du|$. Let $h$ be a harmonic
function in $\Omega$ which extends $\varphi$. Then $u\leq{h}$.
Besides, we also have $u\geq\underline{u}$ by minimum principle and
the assumption of Theorem 1.1. Here we used the Comparison Principle for fully nonlinear equations in the book of D. Gilberge and N. S. Trudinger \cite{gilbarg2001elliptic}(Theorem 17.1 of Page 443) on $F(x,u)=\det(u_{i\overline{j}})-f(x,u)$ and the assumption of $f_u=\frac{\pt{f}}{\pt{u}}\geq0$. Hence,
$$\max_{\pt{\Omega}}|Du|\leq\max_{\pt\Omega}\{|Dh|,D{\underline{u}}|\}.$$
Thus Theorem \ref{thm31} is proved.
\end{proof}
{\bf  Step 2\ \ Reduce the global 2nd-order priori estimates to the
boundary ones. }\\

 Note that, to get the second-order priori estimate of $u$, it is sufficient to prove an upper estimate on it. In fact, let $q_l=t+x\cdot i+y\cdot j+z\cdot k$ be one of the quaternionic coordinates. $\triangle{u}\geq0$ and the upper estimates on the second derivatives of the form $D^2u$ imply the lower estimates on them. The estimate on the mixed derivatives also can be obtained easily since $$2u_{tx}=(\frac{\pt}{\pt{t}}+\frac{\pt}{\pt{x}})^2u-u_{tt}-u_{xx}.$$
 Hence we only need to prove an upper estimate of $D^2u$ on $\pt\Omega$ because of the following lemma:
\begin{lem}\label{lem33} For a constant C only depending on
$\|f\|_{C^0}$, $\|f\|_{C^1}$,  $\|f\|_{C^2}$ and $\Omega$, we have
$$\max_{\overline\Omega}D^2u\leq\max_{\pt\Omega}D^2u+C.$$
\end{lem}
\begin{proof}  From Lemma 7.5 in \cite{alesker2003quaternionic} or Theorem1.1.17 in \cite{alesker2003non}, we know
 \begin{equation}\label{eq:3}
L_0(D^2u)\geq D^2(\log f).
\end{equation} This implies
 \begin{align*}
 L(D^2u+e^{\lambda|q|^2})&\geq D^2(\log f)-f^{-1}{f_u}(D^2u)+\lambda e^{\lambda|q|^2}\sum{u^{ii}}-f^{-1}{f_u}e^{\lambda|q|^2}\\
 &=-f^{-2}{|Df|^2}+f^{-1}(D(f_qDq)+D(f_u)Du)+e^{\lambda|q|^2}(\lambda\sum{u^{ii}}-f^{-1}{f_u})\\
 &\geq-C+e^{\lambda|q|^2}(n\lambda f^{-\frac{1}{n}}-f^{-1}{f_u})>0,\ \ \text{for a large}\ \ \lambda.
 \end{align*}
By maximum principle, we complete the proof.
\end{proof}
{\bf  Step 3\ \ Prove the boundary estimates for second derivatives. }\\

 In this step, we need to derive priori estimates for three kinds of second derivatives: pure tangential derivatives,
 mix derivatives and pure normal derivatives. Since the last ones are more
 complicated, we treat them in the next section. From now on, we will denote the quaternionic units as follows:
$$e_0=1,\ \ \ e_1=i,\ \ \ e_2=j,\ \ \ e_3=k.$$ Fix an arbitrary point $P\in\pt\Omega$, we can choose such a coordinate system $(q_1,\cdots,q_n)$ near this point that the inner normal to $\pt\Omega$ at $P$ coincides with the axis $x^0_n$. we can also assume $P$ to be the origin. For the sake of convenience, we set $$t_1=x^0_1,t_2=x^1_1,t_3=x^2_1,\cdots,t_{4(n-1)}=x^3_{n-1},t_{4n-3}=x^1_n,t_{4n-2}=x^2_n,t_{4n-1}=x^3_n,t_{4n}=x^0_n,$$
and $$t=(t',t_{4n}),\ \ \
t'=(t_1,\cdots,t_\alpha,\cdots,t_\beta,\cdots,t_{4n-1}),\ \ \
\alpha,\beta=1,\cdots,4n-1.$$
\\[8pt]
 {\bf{Part I\ \ \ $|u_{t_{\alpha}t_{\beta}}|(0)<C\sim\|u\|_{C^1},\|\underline{u}\|_{C^1},\|\underline{u}\|_{C^2},\Omega$. }}\\

 As in \cite{caffarelli1985dirichlet}\cite{guan2010complex}, we can write $u-\underline{u}=\tau\sigma$, where $\tau$ is a smooth function and $\sigma$ is the defining function of $\Omega$ with $|\nabla\sigma|=1$. From
 \begin{equation}
 u_{t_it_j}(0)=\underline{u}_{t_it_j}(0)+\tau(0)\sigma_{t_it_j}(0),
 \end{equation}
  \begin{equation}
 (u-\underline{u})_{x_n^0}=\tau_{x_n^0}\sigma+\tau\sigma_{x_n^0},\ \ \ \ \ \ (u-\underline{u})_{x_n^0}(0)=-\tau(0),
 \end{equation}
 we have $|u_{t_{\alpha}t_{\beta}}|<C\sim\|u\|_{C^1},\|\underline{u}\|_{C^1},\|\underline{u}\|_{C^2},\Omega$.\\
 \\[8pt]
 {\bf{Part II}\ \ \ $|u_{t_{\alpha}x^0_n}|(0)<C\sim\|u\|_{C^1},\|\underline{u}\|_{C^1},\|\underline{u}\|_{C^2},\|\underline{u}\|_{C^3},\|f\|_{C^0},\|f\|_{C^1} ,\Omega$. }\\

 Firstly, consider
 \begin{equation}
 h=\pm{T(u-\underline{u})}+\sum\limits_{l=1}^3(u_{x^l_n}-\underline{u}_{x^l_n})^2,\ \ T=\frac{\pt}{\pt{t_{\alpha}}}-\frac{\sigma_{t_\alpha}}{\sigma_{x_n^0}}\frac{\pt}{\pt{x_n^0}}.
 \end{equation}
 By straightforward calculations, we have
 \begin{align*}
 |L(T(u-\underline{u}))|&\leq C+Cnf^{-1}\det(I, \pt^2u[n-1])\\
 &+nf^{-1}\sum_{l=1}^3\det\left(\left((u-\underline{u})_{x_n^l\overline{i}}(u-\underline{u})_{x_n^lj}\right),\pt^2u[n-1]\right)
 \end{align*}
and
 \begin{align*}
 L((u_{x^l_n}-\underline{u}_{x^l_n})^2)
 &\geq-C-Cnf^{-1}\det(I,\pt^2u[n-1])\\
 &+2nf^{-1}\det\left(\left((u-\underline{u})_{x_n^l\overline{i}}(u-\underline{u})_{x_n^lj}\right),\pt^2u[n-1]\right).
 \end{align*}
So we obtain
\begin{align*}
 Lh&\geq-C-Cnf^{-1}\det(I, \pt^2u[n-1])\\
 &+nf^{-1}\sum_{l=1}^3\det\left(\left((u-\underline{u})_{x_n^l\overline{i}}(u-\underline{u})_{x_n^lj}\right),\pt^2u[n-1]\right).
\end{align*}
But the third summand is non-negative. Hence we get
$$Lh\geq-C-Cnf^{-1}\det(I, \pt^2u[n-1]),$$
where the constant $C$ only depends on
$\|u\|_{C^1},\|\underline{u}\|_{C^1},\|\underline{u}\|_{C^2},\|\underline{u}\|_{C^3},\|f\|_{C^0},\|f\|_{C^1}$
and $\Omega$. \\

 Secondly, set $\tilde{w}=h-B|q|^2$ with $B$ to be determined. It is
 clear that
 \begin{equation}
 L{\tilde{w}}=Lh-BL(|q|^2)\geq-C-(B+C)nf^{-1}\det(I, \pt^2u[n-1])-BC,
 \end{equation}
where the value of the constant $C$ might be different from the
previous one.

 Assuming $\delta<1$, we denote $\Omega\cap\mathbf{B}_{\delta}(P)$ by $\Omega_{\delta}$. On $\pt\Omega\cap\mathbf{B}_{\delta}(P)$, we have $T(u-\underline{u})=T(\tau\sigma)\equiv0$.
 This implies
 \begin{equation}\label{eq:33}
 |(u-\underline{u})_{x_k^l}|=|-(u-\underline{u})_{x_n^0}(\sigma_{x_n^0})^{-1}\sigma_{x_k^l}|\leq{C|q|},
 \end{equation}
and then we get $h\leq{C|q|^2}$. On
$\Omega\cap\mathbf{B}_{\delta}(P)$, it is clear $|h|\leq{C}$.

To sum up, we obtain
 \begin{equation}
\begin{cases}
  |h|\leq{C|q|^2},\ \ \ \text{on}\ \ \pt\Omega\cap{\bf{B}}_\delta(P);  \\
  |h|\leq{C},\ \ \ \text{on}\ \ \Omega\cap\overline{{\bf{B}}_\delta(P)}.  \ \ \
\end{cases}
\end{equation}\ \

 Last, suppose $d=d(q)$ to be the distance from point $q$ to the boundary $\pt\Omega$. Let us take $w=\tilde{w}+A(\underline{u}-u-td+\frac{N}{2}d^2)$ as our auxiliary function. Next, we will choose $t,N,\delta$ to reach $Lw>0$.
 By direct calculation and the properties of the mixed discriminant, we can deal with the items in the bracket above as follows:
\begin{align*}
 L(\underline{u}-u)
 &\geq\varepsilon_0nf^{-1}\det(I,\pt^2u[n-1])-n-f^{-1}f_u(\underline{u}-u)\\
 &\geq\varepsilon_0\sum{u^{i\overline{i}}}-C,
\end{align*}
where $\varepsilon_0$ is such a constant that $\underline{u}$ satisfies $(\underline{u}_{i\overline{j}})\geq\varepsilon_0I$.
\begin{align*}
 tL(d)&=tnf^{-1}\det(\pt^2d,\pt^2u[n-1])-tf^{-1}f_ud\\
 &\leq tC\sum{u^{i\overline{i}}}+tCd\leq Ct(d+\sum{u^{i\overline{i}}}).
\end{align*}
\begin{align*}
 \frac{N}{2}L(d^2)&=\frac{N}{2}nf^{-1}\det(2d(d_{i\overline{j}})+(d_id_{\overline{j}}+d_{\overline{j}}d_i),\pt^2u[n-1])-\frac{N}{2}f^{-1}f_ud^2\\
 &=Nnf^{-1}\det(d(d_{i\overline{j}}),\pt^2u[n-1])+N\sum u^{i\overline{i}}d_id_{\overline{i}}-\frac{N}{2}f^{-1}f_ud^2\\
 &\geq-NdC\sum{u^{i\overline{i}}}+\frac{N}{\lambda_n}|\nabla d|^2-\frac{N}{2}f^{-1}f_ud^2\\
 &\geq-NdC\sum{u^{i\overline{i}}}+\frac{N}{2\lambda_n}-\frac{N}{2}f^{-1}f_ud^2,
\end{align*}
where $\lambda_k$ denote the eigenvalues of the matrix $(u_{i\overline{j}})$ and $0<\lambda_1\leq\cdots\leq\lambda_n$. Here we firstly used the fact that the value of the mixed discriminant is pointwisely invariant under orthonormal transition, which comes from Theorem 2.3 and 2.8. Then by Propositon 2.6, we can make the matrix $(\pt^2u)$ diagonal and get the second equality. In $\Omega\cap\mathbf{B}_{\delta}(P)$, we have
$(d_{ij})\geq-CI$. By Propositon 2.13 and Theorem 2.15 (1), we have $$\det(d(d_{i\overline{j}}),\pt^2u[n-1])\geq-C\sum{u^{i\overline{i}}}.$$
Since $|\nabla d|=1$ on $\pt\Omega$, we can choose $\delta$ small such that $|\nabla d|^2\geq\frac{1}{2}$ in $\Omega\cap\mathbf{B}_{\delta}(P)$.
Actually, $|\nabla d|=1$ in a small neighborhood of the boundary. See the book of D. Gilberge and N. S. Trudinger \cite{gilbarg2001elliptic}(Elliptic Partial Differential Equations of Second Order,Page 355). In this paper, we do not need so strong result. The fact $|\nabla d|^2\geq\frac{1}{2}$ is enough. So, $$ N\sum u^{i\overline{i}}d_id_{\overline{i}}\geq\frac{N}{\lambda_n}\sum d_id_{\overline{i}}=\frac{N}{\lambda_n}|\nabla d|^2\geq\frac{N}{2\lambda_n}.$$
Denoting $v=\underline{u}-u-td+\frac{N}{2}d^2$, those inequalities above imply
$$Lv\geq-C_0-C_1(t+Nd)d+[\varepsilon_0-C_1(t+Nd)]\sum{u^{i\overline{i}}}+\frac{N}{\lambda_n}.$$
Note that
$$\frac{\varepsilon_0}{4}\sum{u^{i\overline{i}}}+\frac{N}{\lambda_n}
\geq\frac{\varepsilon_0}{4}\sum_{i=1}^{n-1}{u^{i\overline{i}}}+\frac{N}{\lambda_n}
\geq n(\frac{\varepsilon_0}{4})^{\frac{n-1}{n}}(N\lambda^{-1}_1\cdots\lambda^{-1}_n)^{\frac{1}{n}}
\geq n\frac{\varepsilon_0}{4}f^{-\frac{1}{n}}N^{\frac{1}{n}}=C_2N^{\frac{1}{n}},$$
where we have used $\varepsilon_0<1$ and the inequality of arithmetic and geometric means
$$\sum a_i\geq n(\prod a_i)^{\frac{1}{n}}.$$
Till here, we do not have any restriction on $N$, and the positive constant $\varepsilon_0<1$ is fixed which is determined by $(\underline{u}_{i\overline{j}})\geq\varepsilon_0I$.
Now, we firstly choose $N$ large such that $$C_2N^{\frac{1}{n}}\geq C_0+\frac{3\varepsilon_0}{4}.$$
For such $N$, we choose
$t,\delta(<1)$ so small that $$C_1t<\frac{\varepsilon_0}{4}, C_1Nd<\frac{\varepsilon_0}{4}.$$Then
\begin{align*}
Lv&\geq-C-\frac{\varepsilon_0}{2}d+\frac{\varepsilon_0}{4}\sum{u^{i\overline{i}}}+CN^{\frac{1}{n}}\\
&\geq\frac{\varepsilon_0}{4}(\sum{u^{i\overline{i}}}+1).
\end{align*}

 On one hand, since
 \begin{equation*}
\begin{cases}
  v=0,\ \ \ \text{on}\ \ \pt\Omega\cap{\bf{B}}_\delta(P);  \\
  v\leq-td+\frac{N}{2}d^2<0,\ \ \ \text{on}\ \ \Omega\cap\pt{{\bf{B}}_\delta(P)},  \ \ \
\end{cases}
\end{equation*}
we can choose $B$ so large that
 \begin{equation}
\begin{cases}
  w\leq C|q|^2-B|q|^2<0,\ \ \ \text{on}\ \ \pt\Omega\cap{\bf{B}}_\delta(P);  \\
  w\leq C-B|q|^2<C-B\delta^2<0,\ \ \ \text{on}\ \ \Omega\cap\pt{{\bf{B}}_\delta(P)}.  \ \ \
\end{cases}
\end{equation}

On the other hand, we can also choose such a constant $A\gg B$ that
\begin{equation}
Lw\geq-C-BC+\frac{A\varepsilon_0}{4}+(\frac{A\varepsilon_0}{4}-B-C)\sum{u^{i\overline{i}}}>0,\
\ \text{in}\ \ \Omega_\delta.
\end{equation} By maximum principle,
we get$$w\leq0,\ \ \text{on}\ \ \overline{\Omega_\delta},\ \
w(0)=0.$$ So
$$|(T(u-\underline{u}))_{x_n^0}|(0)\leq|Av_{x_n^0}|(0)\leq C.$$

\section{the second order normal derivatives}
 In this section, our goal is
\begin{equation}\label{eq:41}
|u_{x^0_nx^0_n}|(0)<C\sim\|u\|_{C^1},\|\underline{u}\|_{C^1},\|\underline{u}\|_{C^2},\|\underline{u}\|_{C^3},\|f\|_{C^0},\|f\|_{C^1} ,\Omega
 \end{equation}

 To reach this goal,in $\mathbb{C}^n$ case, B. Guan\cite{guan2010complex} used
 \begin{align*}
\widetilde{h}&=(u_{y_n}-\varphi_{y_n})^2-\widetilde{\Phi}\\
&=(u_{y_n}-\varphi_{y_n})^2-\varphi_{j\overline{k}}\zeta_j\zeta_{\overline{k}}+(u-\varphi)_{x_n}\sigma_{j\overline{k}}\zeta_j\zeta_{\overline{k}}
+u_{1\overline{1}}(0)\\
\widetilde{w}&=-Av+B|q|^2-\widetilde{h},
\end{align*}
where $v$ is same as the one in section 3. By choosing a appropriate vector $\zeta=(\zeta_1,\cdots,\zeta_n)$ and the maximum priciple, he get $\eqref{eq:41}$. However, in $\mathbb{H}^n$, we need to modify $\widetilde{h}$ as
$$
h'=\sum_{k=1}^n\sum_{l=1}^3(u_{x_k^l}-\underline{u}_{x_k^l})^2+
\frac12\sum[(\underline{u}-u)_{k}\sigma_{\overline{k}}+\sigma_{k}(\underline{u}-u)_{\overline{k}}]
\xi_i\sigma_{i\overline{j}}\xi_{\overline{j}}-\sum\xi_i\underline{u}_{i\overline{j}}\xi_{\overline{j}}
+u_{1\overline1}(0),$$

\begin{rmk} Here we should use $u_{1\overline1}(0)$ as Bo Guan do in \cite{guan2010complex}, which is defined later. In $\mathbb{C}^n$ case, our auxiliary function $h'$ is reduced to
\begin{equation}
h'=\sum_{k=1}^n(u_{y_k}-\underline{u}_{y_k})^2+
\sum(\underline{u}-u)_{k}\sigma_{\overline{k}}\xi_i\sigma_{i\overline{j}}\xi_{\overline{j}}
-\sum\xi_i\underline{u}_{i\overline{j}}\xi_{\overline{j}}+u_{1\overline1}(0),
\end{equation}
where $z_k=x_k+iy_k$.
\end{rmk}

Now, let's start to prove $\eqref{eq:41}$. First, we only need to prove $$0\leq u_{n\overline{n}}(0)=u_{x_n^0x_n^0}(0)+\sum_{1\leq l\leq3}u_{x_n^lx_n^l}(0)\leq C.$$
To prove this, we need the lemma below:
\begin{lem}\label{lem41} If $\sum{u_{q_a\overline{q}_b}\xi_a\overline{\xi}_b}\geq C|\xi|^2$ for any $\xi\in\mathbb{H}^{n-1}$ and $1\leq{a,b}\leq{n-1}$. Then
$$0\leq{u_{n\overline{n}}(0)}\leq C.$$
\end{lem}
\begin{proof}
 By definition, it is clear $$\det(u_{i\overline{j}})(0)=\det(u_{a\overline{b}})u_{n\overline{n}}(0)+R=f,$$ where $R$ denotes the remainder terms and $u_{a\overline{b}}$ denotes $\frac{\pt^2u}{\pt{q_a}\pt{\overline{q}_b}}$. Then
 $$u_{n\overline{n}}(0)=\frac{f-R}{\det(u_{a\overline{b}})}\leq C.$$
 \end{proof}
 By this lemma, it suffices to prove
$$m_0\triangleq\min_{P\in\pt\Omega}\min_{\stackrel{\xi\in T_P^{\mathbb{H}}\pt\Omega}{|\xi|=1}}\xi_iu_{i\overline{j}}\overline{\xi}_j\geq C.$$
We can choose coordinates such that $m_0$ attains at the origin $P\in\pt\Omega$ when
$\xi=(1,0,\cdots,0)$. Then we only need to prove
$$m_0=u_{1\overline1}(0)\geq C>0.$$
We will use the auxiliary function $h'$ to get this goal in the same coordinates. By $\eqref{eq:47}$, $u_{1\overline1}(0)$ is under control.

Since $u-\underline{u}=\tau\sigma$, where $\tau$ is a smooth function
and $\sigma$ is the defining function of $\Omega$ with
$|\nabla\sigma|=1$, we get some basic formulas as follows:
 \begin{equation}\label{eq:43}
 (u-\underline{u})_{i}=\tau_{i}\sigma+\tau\sigma_{i},\ \ \
 \sum(u-\underline{u})_{i}\sigma_{\overline{i}}=\tau_i\sigma_{\overline{i}}\sigma+\tau\sigma_i\sigma_{\overline{i}};
 \end{equation}
 \begin{equation}\label{eq:44}
 (u-\underline{u})_{i\overline{j}}=\tau_{i\overline{j}}\sigma+\sigma_{\overline{j}}\tau_i+\tau_{\overline{j}}\sigma_i+\tau\sigma_{i\overline{j}}.
 \end{equation}

From $\eqref{eq:43}$ and $\eqref{eq:44}$, we have
 \begin{equation}\label{eq:45}
\tau=\sum\limits_{k}\frac{(u-\underline{u})_{k}\sigma_{\overline{k}}+\sigma_{k}(u-\underline{u})_{\overline{k}}}{2|\nabla\sigma|^2} \ \ \text{on} \ \ \pt\Omega,
\end{equation} and
 \begin{equation}\label{eq:46}
\xi_i(u-\underline{u})_{i\overline{j}}\xi_{\overline{j}}=\tau\xi_i\sigma_{i\overline{j}}\xi_{\overline{j}} \ \ \text{for any} \ \ q\in\pt\Omega,\xi\in T_q^{\mathbb{H}}\pt\Omega.
\end{equation}
So
\begin{equation}\label{eq:47}
u_{1\overline1}(0)=\underline{u}_{1\overline1}(0)-(u-\underline{u})_{x_n^0}(0)\sigma_{1\overline1}(0).
\end{equation}

We can assume
$u_{1\overline1}(0)<\frac12\underline{u}_{1\overline1}(0)$,
otherwise, $m_0\geq C>0$ can be obtained immediately. Using this
condition, we have
$(u-\underline{u})_{x_n^0}\sigma_{1\overline1}(0)=\underline{u}_{1\overline1}(0)-u_{1\overline1}(0)\geq\frac12\underline{u}_{1\overline1}(0).$
Then
 \begin{equation}
\sigma_{1\overline1}(0)\geq C>0.
\end{equation}
 Now, let us consider $w'=h'-B|q|^2+Av$, where
\begin{equation}
h'=\sum_{k=1}^n\sum_{l=1}^3(u_{x_k^l}-\underline{u}_{x_k^l})^2+\Phi,
\end{equation}
\begin{equation}\label{eq:410}
\Phi=
\frac12\sum[(\underline{u}-u)_{k}\sigma_{\overline{k}}+\sigma_{k}(\underline{u}-u)_{\overline{k}}]
\xi_i\sigma_{i\overline{j}}\xi_{\overline{j}}-\sum\xi_i\underline{u}_{i\overline{j}}\xi_{\overline{j}}+u_{1\overline1}(0).
\end{equation}
In $\eqref{eq:410}$, we mean the indices $i,j,k$ in the sum run from $1$ to $n$ as elsewhere, and the vector $\xi$ depends on the point $q\in\overline{\Omega_\delta}$ to be determined below.

Thanks to  Part II, we only need to prove
\\[8pt]
 (J1)\ \ $Lh'\geq-C(1+\sum_iu^{i\overline{i}}),\ \ \text{in} \ \ \Omega_\delta$ ;
 \\[8pt]
 (J2)\ \ $h'\leq{C|q|^2},\ \ \ \text{on}\ \ \pt\Omega\cap{\bf{B}}_\delta(P)$;
 \\[8pt]
 (J3)\ \ $|h'|\leq{C},\ \ \ \text{on}\ \ \Omega\cap\overline{{\bf{B}}_\delta(P)}$.\\

In fact, as in Part II, by these inequalities and maximum principle
we can get $\Phi_{x_n^0}(0)\leq -Av_{x_n^0}(0)$. Here we notice that $w'(0)=h'(0)=\Phi(0)=0$. Indeed, by $\eqref{eq:45}$ and $\eqref{eq:46}$, we have
\begin{align*}
\Phi(0)&=\sum\xi_i(\underline{u}-u)_{i\overline{j}}\xi_{\overline{j}}-\sum\xi_i\underline{u}_{i\overline{j}}\xi_{\overline{j}}+u_{1\overline1}(0)\\
&=-\sum\xi_iu_{i\overline{j}}\xi_{\overline{j}}(0)+u_{1\overline1}(0)=0,
\end{align*}
the last equality holds provided $\xi_i(0) =0$ for $i > 1$. Actually, this fact is included in $\eqref{eq:411}$.

 Equivalently, there
exists a constant $C>0$ satisfies
$$u_{x_n^0x_n^0}(0)\sigma_{1\overline1}(0)\leq A(u-\underline{u})_{x_n^0}(0)+Atd_{x_n^0}(0)+C\leq C,\ \ u_{x_n^0x_n^0}(0)\leq C,$$
here $C$ are not same and we have chosen
 \begin{equation*}
 \xi_i=
\begin{cases}
  -\frac{\sigma_n}{\omega},\ \ \ \text{if}\ \ \ i=1;  \\
  0,\ \ \ \text{if}\ \ \ 2\leq i\leq{n-1};\\
  \frac{\sigma_1}{\omega},\ \ \ \text{if}\ \ \ i=n,  \ \ \
\end{cases}
\end{equation*}
where $\omega=|\xi|_{\mathbb{H}^n}$.

Since (J3) is obviously satisfied, we will check (J1) and (J2). By
$\eqref{eq:46}$, we have
\begin{equation}\label{eq:411}
\Phi=-\sum_{i,j}\xi_iu_{i\overline{j}}\xi_{\overline{j}}+u_{1\overline1}(0)\leq0,\
\ \ \text{on}\ \ \ \pt\Omega\cap{\bf{B}}_\delta(P).
\end{equation}
Therefore, (J2) follows from $\eqref{eq:33}$.

 In the rest of this section, we derive (J1) to finish Step 3. For simplicity, we set
$\mu=\sum\xi_i\sigma_{i\overline{j}}\xi_{\overline{j}},
\mu_k^0=\sigma_{x_k^0}\mu, \mu_k^l=\sigma_{x_k^l}\mu$, we get
 \begin{align*}
 L\Phi&=\frac12\sum L((\underline{u}-u)_{k}\sigma_{\overline{k}}\mu+\sigma_{k}(\underline{u}-u)_{\overline{k}}\mu)-L(\sum\xi_i\underline{u}_{i\overline{j}}
\xi_{\overline{j}})\\
 &=\frac12\sum L((\underline{u}_{k}\sigma_{\overline{k}}+\sigma_{k}\underline{u}_{\overline{k}})\mu)-\frac12\sum L((u_{k}\sigma_{\overline{k}}+\sigma_{k}u_{\overline{k}})
\mu)-L(\sum\xi_i\underline{u}_{i\overline{j}}\xi_{\overline{j}})\\
&=E+F+G,
\end{align*}
where
\begin{align*}
 E&=\frac12\sum L((\underline{u}_{k}\sigma_{\overline{k}}+\sigma_{k}\underline{u}_{\overline{k}})\mu),\\
 F&=-\frac12\sum L((u_{k}\sigma_{\overline{k}}+\sigma_{k}u_{\overline{k}})\mu),\\ G&=L(\sum\xi_i\underline{u}_{i\overline{j}}\xi_{\overline{j}}).
\end{align*}
It is clear
\begin{equation}
 E+G\geq-C(1+\sum{u^{i\overline{i}}}).
 \end{equation}
  As for $F$, we can write it as
\begin{align*}
 F&=-\sum_{k=1}^n L(u_{x_k^0}\mu_k^0+\sum_{l=1}^3\mu_k^lu_{x_k^l})\\
&\geq-C-nf^{-1}\sum_{k=1}^n\det\left((u_{x_k^0i\overline{j}}\mu_k^0+{\mu_k^0}_{i\overline{j}}u_{x_k^0}),\pt^2u[n-1]\right)\\
&-nf^{-1}\sum_{k=1}^n\det\left((u_{x_k^0\overline{j}}{\mu_k^0}_i+{\mu_k^0}_{\overline{j}}u_{x_k^0i}),\pt^2u[n-1]\right)\\
&-nf^{-1}\sum_{k=1}^n\sum_{l=1}^3\det\left((u_{x_k^li\overline{j}}\mu_k^l+{\mu_k^l}_{i\overline{j}}u_{x_k^l}),\pt^2u[n-1]\right)\\
&-nf^{-1}\sum_{k=1}^n\sum_{l=1}^3\det\left((u_{x_k^l\overline{j}}{\mu_k^l}_i+{\mu_k^l}_{\overline{j}}u_{x_k^li}),\pt^2u[n-1]\right).
\end{align*}
We denote the last four terms in the right of the inequality above
by $F_1,F_2,F_3,F_4$ one by one. Thus it is easy to check
$F_1+F_3\geq-C(1+\sum{u^{i\overline{i}}})$.

By Claim \ref{clm214}, we get
\begin{align*}
 F_2&\geq -nf^{-1}\sum_{k=1}^n\det\left(({\mu_k^0}_{\overline{j}}u_{i\overline{k}})+({\mu_k^0}_{\overline{j}}u_{i\overline{k}})^*,\pt^2u[n-1]\right)\\
 &+nf^{-1}\sum_{k=1}^n\sum_{l=1}^3\det\left(({\mu_k^0}_{\overline{j}}e_lu_{x_k^li})+({\mu_k^0}_{\overline{j}}e_lu_{x_k^li})^*,\pt^2u[n-1]\right)\\
 &=-2\sum_{k=1}^n\mathrm{Re}({\mu_k^0}_{\overline{k}})+nf^{-1}\sum_{k=1}^n\sum_{l=1}^3\det\left(({\mu_k^0}_{\overline{j}}e_lu_{x_k^li})+({\mu_k^0}_{\overline{j}}e_lu_{x_k^li})^*,\pt^2u[n-1]\right),
\end{align*}

For $\forall\varepsilon_0>0$, from Corollary \ref{cor217}, we obtain
\begin{equation}
\begin{split}
F_2+F_4
&\geq -C-nf^{-1}\sum_{k=1}^n\sum_{l=1}^3\det\left(({\mu_k^l}_{\overline{j}}u_{x_k^li})+({\mu_k^l}_{\overline{j}}u_{x_k^li})^*,\pt^2u[n-1]\right)\\
&+nf^{-1}\sum_{k=1}^n\sum_{l=1}^3\det\left(({\mu_k^0}_{\overline{j}}e_lu_{x_k^li})+({\mu_k^0}_{\overline{j}}e_lu_{x_k^li})^*,\pt^2u[n-1]\right)\\
&\geq-C-nf^{-1}(1-\varepsilon_0)^{-1}\sum_{k=1}^n\sum_{l=1}^3\det(({\mu_k^l}_{i}{\mu_k^l}_{\overline{j}}),\pt^2u[n-1])\\
&-nf^{-1}(1-\varepsilon_0)\sum_{k=1}^n\sum_{l=1}^3\det((u_{x_k^li}u_{x_k^l\overline{j}}),\pt^2u[n-1]))\\
&-nf^{-1}(1-\varepsilon_0)^{-1}\sum_{k=1}^n\sum_{l=1}^3\det(({\mu_k^0}_{i}{\mu_k^0}_{\overline{j}}),\pt^2u[n-1])\\
&-nf^{-1}(1-\varepsilon_0)\sum_{k=1}^n\sum_{l=1}^3\det((u_{x_k^li}u_{x_k^l\overline{j}}),\pt^2u[n-1])\\
&\geq
-C(1+\sum_iu_{i\overline{i}})-(2-2\varepsilon_0)nf^{-1}\sum_{k=1}^n\sum_{l=1}^3\det((u_{x_k^li}u_{x_k^l\overline{j}}),\pt^2u[n-1]),
\end{split}
\end{equation}
\begin{equation}
\begin{split}
L((u_{x_k^l}-\underline{u}_{x_k^l})^2)&\geq -C+2(u_{x_k^l}-\underline{u}_{x_k^l})nf^{-1}\det((u_{x_k^l}-\underline{u}_{x_k^l})_{i\overline{j}},\pt^2u[n-1])\\
&+2nf^{-1}\det((u_{x_k^l}-\underline{u}_{x_k^l})_i(u_{x_k^l}-\underline{u}_{x_k^l})_{\overline{j}},\pt^2u[n-1])\\ &\geq-C(1+\sum_iu_{i\overline{i}})+2nf^{-1}\det((u_{x_k^li}u_{x_k^l\overline{j}}),\pt^2u[n-1])\\
&+2nf^{-1}\det((u_{x_k^li}{\underline{u}}_{x_k^l\overline{j}}+{\underline{u}}_{x_k^li}u_{x_k^l\overline{j}}),\pt^2u[n-1])\\
&\geq-C(1+\sum_iu_{i\overline{i}})+2nf^{-1}\det((u_{x_k^li}u_{x_k^l\overline{j}}),\pt^2u[n-1])\\
&+2nf^{-1}{\varepsilon_0}^{-1}\det(({\underline{u}}_{x_k^li}{\underline{u}}_{x_k^l\overline{j}}),\pt^2u[n-1])\\
&+2nf^{-1}\varepsilon_0\det((u_{x_k^li}u_{x_k^l\overline{j}}),\pt^2u[n-1])\\
&\geq-C(1+\sum_iu_{i\overline{i}})+(2-2\varepsilon_0)nf^{-1}\det((u_{x_k^li}u_{x_k^l\overline{j}}),\pt^2u[n-1]).
\end{split}
\end{equation}
Hence, by the arguements in \cite{caffarelli1984dirichlet}\cite{caffarelli1985dirichlet} and standard elliptic theory, we complete the proof of Theorem \ref{thm11}.

\section{Construction of a subsolution}
 In this section, we will construct a subsolution to \eqref{eq15} in a strictly pseudoconvex domain by the method in \cite{caffarelli1985dirichlet}. To prove Proposition \ref{prop12}, we first
show that under the assumption of Proposition \ref{prop12}, if $u\leq m$,
then there exists a constant C=C(m) satisfies
 $$f(q,u,p)\leq C(1+|p|^n) \ \ \text{for}\ \ \ q\in\overline\Omega.$$
 Indeed, if $\omega\leq m$, then
 \begin{align*}
 f(q,\omega,\eta)&=f(q,0,0)+\int_0^1\frac{d}{dt}f(q,t\omega,t\eta)dt\\
 &\leq C+m\int_0^1f_u(q,t\omega,t\eta)dt+\sum_{i=1}^{4n}|p_i|\int_0^1|f_{p_i}(q,t\omega,t\eta)|dt\\
 &\leq C(1+\max_{0\leq t\leq1}f^{1-\frac1n}(q,t\omega,t\eta)(C+|p|)).
 \end{align*}
Setting $\Lambda=\max\limits_{\substack{{|\eta|\leq|p|}\\ {\omega\leq m}}}f(q,\omega,\eta)$, we get $$f(q,u,p)\leq\Lambda\leq C(1+|p|^n).$$
 Now, we define $$\underline{u}=\varphi+s(e^{kr}-1),\ \ \ k,s>0,$$ where $r$ is a strictly  psh defining function for $\Omega$. Extending $\varphi$ as a psh $C^\infty$ function in $\overline\Omega$, we get $$\det(\underline{u}_{i\overline{j}})\geq(ske^{kr})^n\det(r_{i\overline{j}}+kr_{\overline{j}}r_i).$$
Let $\alpha>0$ be such that $(r_{i\overline{j}})\geq\alpha I$. Then $$\det(r_{i\overline{j}}+kr_{\overline{j}}r_i)\geq\alpha^{n-1}(\alpha+k|\nabla r|^2).$$
To check this at a point $z^0\in\overline\Omega$ choose coordinate such that $r_i(z^0)=0$ for $i<n$, then $|\nabla r(z^0)|=|r_n(z^0)|$ and the above inequality follows. Now we have $$\det(\underline{u}_{i\overline{j}})\geq(sk\alpha e^{kr})^n(1+\frac{k}{\alpha}|\nabla r|^2),$$ and
$$|\nabla\underline{u}|\leq\max|\nabla\varphi|+ske^{kr}|\nabla r|,$$
$$|\nabla\underline{u}|^n\leq C_1+C_2(ske^{kr})^n|\nabla r|^n.$$
Choose $k$ so large that $$C(m)C_2|\nabla r|^n\leq k\alpha^{n-1}|\nabla r|^2,$$
then we can choose $s$ so that
\begin{align*}
f(q,\underline{u},\nabla\underline{u})&\leq C(m)(1+|\nabla\underline{u}|^n)
\leq C(m)(1+C_1+C_2(ske^{kr})^n|\nabla r|^n)\\
&\leq C(m)(1+C_1)+(sk\alpha e^{kr})^n\frac{k}{\alpha}|\nabla r|^2\\
&\leq (sk\alpha e^{kr})^n(1+\frac{k}{\alpha}|\nabla r|^2)\leq\det(\underline{u}_{i\overline{j}}),
\end{align*}
where $m=\max\varphi$. Hence we prove Proposition \ref{prop12}.

\section{Acknowledgements}
 I am very grateful to my supervisor Professor Xinan Ma for his patient instruction and expert guidance. Without his help, the completion of this paper would be impossible. Also, I would like to thank my mentor Professor Jie Qing for his encouragement. Moreover, I'd like to thank my friends Guohuan Qiu and Dekai Zhang for helpful discussions in the seminar on geometric PDEs.

\nocite{*}

\bibliographystyle{amsplain}
\bibliography{reference}

\providecommand{\bysame}{\leavevmode\hbox to3em{\hrulefill}\thinspace}
\providecommand{\MR}{\relax\ifhmode\unskip\space\fi MR }
\providecommand{\MRhref}[2]{%
  \href{http://www.ams.org/mathscinet-getitem?mr=#1}{#2}
}
\providecommand{\href}[2]{#2}
\begin{thebibliography}{10}

\bibitem{aleksandrov1938gemischte}
Alexander~D Aleksandrov, \emph{Die gemischte {D}iskriminanten und die gemischte
  {V}olumina}, Math. Sbornik \textbf{3} (1938), 227--251.

\bibitem{alesker2003non}
Semyon Alesker, \emph{Non-commutative linear algebra and plurisubharmonic
  functions of quaternionic variables}, Bulletin des sciences mathematiques
  \textbf{127} (2003), no.~1, 1--35.

\bibitem{alesker2003quaternionic}
\bysame, \emph{Quaternionic {M}onge-{A}mpere equations}, The Journal of
  Geometric Analysis \textbf{13} (2003), no.~2, 205--238.

\bibitem{alesker2013solvability}
\bysame, \emph{Solvability of the quaternionic {M}onge-{A}mp{\`e}re equation on
  compact manifolds with a flat hyper{K}{\"a}hler metric}, Advances in
  Mathematics \textbf{241} (2013), 192--219.

\bibitem{alesker2013uniform}
Semyon Alesker and Egor Shelukhin, \emph{On a uniform estimate for the
  quaternionic {C}alabi problem}, Israel Journal of Mathematics \textbf{197}
  (2013), no.~1, 309--327.

\bibitem{alesker2010quaternionic}
Semyon Alesker and Misha Verbitsky, \emph{Quaternionic {M}onge-{A}mpere
  equation and {C}alabi problem for {HKT}-manifolds}, Israel Journal of
  Mathematics \textbf{176} (2010), no.~1, 109--138.

\bibitem{aslaksen1996quaternionic}
Helmer Aslaksen, \emph{Quaternionic determinants}, The Mathematical
  Intelligencer \textbf{18} (1996), no.~3, 57--65.

\bibitem{caffarelli1985dirichlet}
Luis Caffarelli, Joseph Kohn, Louis Nirenberg, and Joel Spruck, \emph{The
  {D}irichlet problem for nonlinear second-order elliptic equations. {II}.
  complex {M}onge-{A}mp{\`e}re, and uniformaly elliptic, equations},
  Communications on pure and applied mathematics \textbf{38} (1985), no.~2,
  209--252.

\bibitem{caffarelli1984dirichlet}
Luis Caffarelli, Louis Nirenberg, and Joel Spruck, \emph{The {D}irichlet
  problem for nonlinear second-order elliptic equations {I}.
  {M}onge-{A}mp{\'e}gre equation}, Communications on pure and applied
  mathematics \textbf{37} (1984), no.~3, 369--402.

\bibitem{calabi1954space}
Eugenio Calabi, \emph{The space of {K}$\ddot{a}$hler metrics}, Proceedings of
  the International Congress of Mathematicians \textbf{2} (1954), 206--207.

\bibitem{gilbarg2001elliptic}
David Gilbarg and Neil~S Trudinger, \emph{Elliptic partial differential
  equations of second order}, vol. 224, springer, 2001.

\bibitem{guan1998dirichlet}
Bo~Guan, \emph{The {D}irichlet problem for complex {M}onge-{A}mpere equations
  and regularity of the pluri-complex {G}reen function}, Communications in
  Analysis and Geometry \textbf{6} (1998), no.~4, 687--703, a correction,
  \textbf{8} (2000), 213--218.

\bibitem{guan2010complex}
Bo~Guan and Qun Li, \emph{Complex {M}onge-{A}mp{\`e}re equations and totally
  real submanifolds}, Advances in Mathematics \textbf{225} (2010), no.~3,
  1185--1223.

\bibitem{guan2013dirichlet}
\bysame, \emph{The {D}irichlet problem for a complex {M}onge-{A}mp\`ere
  equation on {H}ermitian manifolds}, Advances in Mathematics \textbf{246}
  (2013), 351--367.

\bibitem{guan2013class}
Bo~Guan and Wei Sun, \emph{On a class of fully nonlinear elliptic equations on
  {H}ermitian manifolds}, pre-print (2013), arXiv:1301.5863[math.AP].

\bibitem{guan2002extremal}
Pengfei Guan, \emph{The extremal function associated to intrinsic norms},
  Annals of mathematics \textbf{156} (2002), no.~1, 197--211.

\bibitem{moore1922determinant}
Eliakim~Hastings Moore, \emph{On the determinant of an hermitian matrix of
  quaternionic elements}, Bull. Amer. Math. Soc \textbf{28} (1922), 161--162.

\bibitem{yau1978ricci}
Shing-Tung Yau, \emph{On the ricci curvature of a compact {K}{\"a}hler manifold
  and the complex {M}onge-{A}mp{\'e}re equation, {I}}, Communications on pure
  and applied mathematics \textbf{31} (1978), no.~3, 339--411.

\end{thebibliography}
\end{document}